\documentclass[11pt]{amsart}
\usepackage{amsmath, amsthm, amssymb, amsfonts}
\usepackage{mathtools}
\usepackage{subcaption}
\usepackage[normalem]{ulem}
\usepackage[numbers]{natbib}
\usepackage{graphicx}
\usepackage{booktabs, array}
\usepackage{xcolor}
\usepackage[foot]{amsaddr}
\usepackage{tabularx}
\usepackage{hyperref}
\usepackage{tikz}    
\usepackage{tikz-cd}    
\usetikzlibrary{arrows.meta}
\hypersetup{
    colorlinks=true,
    citecolor=green!50!black,
    linkcolor=red!80!black,
    filecolor=blue,
    urlcolor=blue!70!black,
}
\usepackage[normalem]{ulem}
\usepackage{comment}
\usepackage[margin=1.5in]{geometry}

\newtheorem{theorem}{Theorem}[section]
\newtheorem{proposition}[theorem]{Proposition}

\newtheorem{lemma}[theorem]{Lemma}
\theoremstyle{definition}
\newtheorem{remark}[theorem]{Remark}
\newtheorem{example}[theorem]{Example}
\newtheorem{definition}[theorem]{Definition}

\newcommand{\arxivonly}[1]{#1}
\newcommand{\journalonly}[1]{}

\newcommand{\Symp}{\mathrm{Symp}}
\newcommand{\vect}{\mathrm{vect}}
\newcommand{\symp}{\mathrm{symp}}
\newcommand{\ham}{\mathrm{ham}}
\newcommand{\Ham}{\mathrm{Ham}}
\newcommand{\Aug}{\mathcal{ FD}}
\newcommand{\Tri}{\mathcal T\!ri}

\title[Generalized Marsden--Weinstein symplectic structures]{Generalized Marsden--Weinstein\\symplectic structures}

\author{Boris Khesin${^{1\journalonly{,*}}}$}
\address{${^1}$Department of Mathematics, University of Toronto, ON M5S 2E4, Canada}
\arxivonly{
\email{khesin@math.toronto.edu}
}
\journalonly{
\thanks{${}^*$Corresponding author, \texttt{klas.modin@chalmers.se}, \href{https://orcid.org/0000-0001-6900-1122}{ORCID:0000-0001-6900-1122}}
}

\author{Klas Modin${^{2}}$}
\address{${^2}$Department of Mathematical Sciences, Chalmers University of Technology and University of Gothenburg, SE-412~96 Gothenburg, Sweden}
\arxivonly{
\email{klas.modin@chalmers.se}
}
\author{Cornelia Vizman${^{3}}$}
\address{${^3}$Department of Mathematics, West University of Timi\c soara, Bd. V.Parvan 4, 300223 Timi\c soara, Romania}
\arxivonly{
\email{cornelia.vizman@e-uvt.ro}
}


\date{\today}

\renewcommand{\O}{\mathcal{O}}

\newcommand{\Sub}{\operatorname{Sub}}
\newcommand{\Emb}{\operatorname{Emb}}
\newcommand{\Diff}{\operatorname{Diff}}

\begin{document}

 \dedicatory{To the memory of Vladimir Igorevich Arnold}

\begin{abstract}
We define a symplectic structure on singular 1-form densities and 
a presymplectic structure on augmented triples, a compressible analog of the 
Marsden--Weinstein symplectic structure on vortex membranes in the incompressible case. 
Equivalence classes of  augmented triples  can be regarded as
special vorticities in compressible fluids, i.e., singular elements of the  dual  to the Lie algebra of the group of all diffeomorphisms of a manifold, while Marsden--Weinstein symplectic structure is related to singular vorticities for the group of volume-preserving diffeomorphisms. This  structure on singular 1-form densities occupies an intermediate position between 
two other settings, which we revisit,  the Weinstein symplectic structure on weighted isotropic submanifolds 
of symplectic manifolds and the Marsden--Weinstein symplectic structure on vortex membranes.
\end{abstract}
\date{}
\maketitle

\arxivonly{
\vspace{-2ex}
}

\journalonly{
\subsection*{Statements and Declarations}
The authors have no relevant financial or non-financial interests to disclose.
Data sharing is not applicable to this article as no datasets were generated or analyzed during the current study.

\subsection*{Acknowledgements}
This work was supported by the Swedish Research Council (grant number 2022-03453), the Knut and Alice Wallenberg Foundation (grant numbers WAF2019.0201 and KAW2020.0287), and the Göran Gustafsson Foundation for Research in Natural Sciences and Medicine.
The computations were enabled by resources provided by Chalmers e-Commons at Chalmers.
The data handling was enabled by resources provided by the National Academic Infrastructure for Supercomputing in Sweden (NAISS), partially funded by the Swedish Research Council through grant agreement no.~2022-06725.
}

\section{Introduction}

The Marsden--Weinstein symplectic structure on embedded closed curves  in $\mathbb R^3$, i.e.,
the so called symplectic structure on knots, arises in many settings in geometry, physics, and PDEs.
For instance, the binormal (or vortex filament) equation, or its equivalent counterpart, the 1D non-linear Schrödinger (NLS) equation, describes a Hamiltonian evolution on the space of singular knots  
with respect to this symplectic structure for the Hamiltonian function given by the curve's length. 

The beauty and simplicity of the Marsden--Weinstein symplectic structure  stimulated various generalizations.
It was extended to the space of vortex membranes in incompressible fluid, the codimension 2 submanifolds 
in a manifold of any dimension equipped with a volume form. There exists an analog of this structure for vortex sheets, understood as codimension 1 submanifolds equipped with a closed 1-form \cite{Kh2012}. 
This symplectic structure 
is  related to an ideal fluid dynamics: it coincides with the Kirillov--Kostant symplectic structure on coadjoint orbits of special ``singular" vorticities in the dual space of the Lie algebra of divergence-free 
vector fields \cite{MW1983}. 
In turn, Weinstein showed that in a similar way one can define a symplectic structure on the space of weighted isotropic submanifolds in a symplectic manifold. The latter can be viewed as singular vorticity elements  
in a symplectic fluid \cite{Kh2012Gafa}.

In this paper we define and 
discuss properties of the symplectic structure on yet another space, that of 
singular vorticities in a compressible fluid, see Section~\ref{sec:triples}.
These are given by the singular 1-form densities defined in  \cite{HV2025}, where also their coadjoint orbits for nowhere vanishing 
1-forms
were studied in detail from the point of view of symplectic reduction and dual pairs. 

Arbitrary singular 1-form densities  can be represented by  (an equivalence class of) augmented triples, consisting of a compact weighted submanifold of any (co)dimension equipped with a 1-form in any manifold. 
As we will show, the symplectic form on 1-form densities, in a sense, occupies an intermediate position between 
the Marsden--Weinstein and Weinstein symplectic structures. 
It naturally coincides with the symplectic structure on special coadjoint orbits of the full diffeomorphism group of the ambient manifold. 
To describe it, we revisit both structures and highlight their similarities and differences, so the paper is also a survey of singular coadjoint orbits in fluids. 

We summarize various singular objects on which symplectic structures are defined below in  Table~\ref{table}.

{\small 
\begin{table}[t]
\centering
\renewcommand{\arraystretch}{1.5}
\setlength{\tabcolsep}{5pt}
\begin{tabular}{>{\raggedright\arraybackslash}p{0.33\textwidth}>{\raggedright\arraybackslash}p{0.25\textwidth}>{\raggedright\arraybackslash}p{0.34\textwidth}}
\toprule
\textbf{Type of symplectic structure and fluid} & \textbf{(Co)dimension of submanifold} & \textbf{Singular current in an ambient manifold} \\
\midrule
Marsden--Weinstein s/s on membranes, incompressible fluid, \S\ref{sec:MW}& codimension 2 & submanifold $\Sigma\subset S$\\ 
Marsden--Weinstein-type s/s on vortex sheets, incompressible fluid, \S\ref{sec:MW}& codimension 1 & pair $(P,\beta)$ 
with 1-form $\beta$ on submanifold $P\subset S$\\
Weinstein s/s on isodrastic submanifolds, symplectic fluid, \S\ref{sec:Weinstein}
& dimension $k\le n$ & weighted isotropic $(\Gamma, \rho)$ in symplectic $(M^{2n}, \omega)$ with $k$-form $\rho$ on $\Gamma$\\
generalized s/s on singular 1-form densities, compressible fluid, \S\ref{sec:triples}
& dimension $k\le n$& singular 1-form density $(\Gamma, \alpha\otimes\rho)$ in any $Q^n$ with $k$-form $\rho$ on $\Gamma$ and 1-form $\alpha\in \Omega^1_\Gamma(Q)$  defined on $\Gamma\subset Q$\\
\bottomrule
\end{tabular}

\caption{Types of singular currents for symplectic structures (s/s) discussed in the paper.}
\label{table}
\end{table}
}

Also the exposition of many structures can be set into a universal framework of mappings of a manifold
with a volume form to a symplectic manifold or dual pairs, described by Donaldson \cite{Do1999} and Marsden--Weinstein \cite{MW1983}. 
In Section \ref{sec:binormal} we discuss open questions related to other groups and possible generalizations of the binormal equation to the setting of triples. Finally, in Appendix~\ref{sect:reduction} we show how the formalize the Marsden--Weinstein  symplectic structures on vortex sheets and vortex membranes within the Donaldson framework of submersions of a manifold with volume forms to a surface. In particular, the Donaldson
 setting is applied here to the case, where the target symplectic manifold is itself infinite-dimensional, the space of membranes equipped with the Marsden--Weinstein   structure. 

\bigskip

\noindent\textbf{Acknowledgements.} The authors are grateful to  Ood Shabtai and Luke Volk for helpful discussions.
B.K. was partially supported by an NSERC Discovery Grant. K.M. was supported by the Swedish Research Council (grant number 2022-03453), the Knut and Alice Wallenberg Foundation (grant numbers WAF2019.0201), and the Göran Gustafsson Foundation for Research in Natural Sciences and Medicine.
C.V. gratefully acknowledges the hospitality and financial support of the Erwin Schrödinger International Institute for Mathematics and Physics,
specifically during the thematic program ``Infinite-dimensional Geometry: Theory and Applications''.


\section{Donaldson's setting}\label{sec:Donaldson}

\subsection{Symplectic structure on the space of maps}
Let $S$ be a compact $k$-dimensional manifold equipped with a volume form $\sigma$, and  let $(M,\omega)$ be a symplectic manifold.
Consider the space $\mathcal M$ of smooth maps $S\to M$ within a fixed homotopy class.
As an infinite-dimensional manifold, the tangent space $T_f\mathcal M$ consists of smooth sections of the pullback bundle $f^*(TM)$, i.e., maps $u\colon S\to TM$ such that $u(s) \in T_{f(s)}M$ for any $s\in S$. We are interested in the local structure of $\mathcal M$ near its nonsingular point.

\begin{definition}[Donaldson~\cite{Do1999}]
The infinite-dimensional  manifold $\mathcal M$ can be equipped with the following symplectic structure
\begin{equation}\label{eq:Omega_f}
    \Omega_f(u,v) 
    = \int_S \omega_{f(s)}(u(s),v(s))\sigma(s).
\end{equation}
\end{definition}

\begin{remark}
It is non-degenerate on $\mathcal M$ since it is defined point-wise and since $\omega$ is non-degenerate
on the manifold $M$. One can also show that it is  closed \cite{Do1999}.  
Indeed, the  symplectic structure on sections of $f^*(TM)$ is the averaged symplectic structure in the fibers weighted by $\sigma$, from which its closedness follows. 
\end{remark}

The group $\mathcal G = \mathrm{Diff}_\sigma(S)$ acts on $\mathcal M$ from the right.
The lifted action on $T\mathcal M$ is also just by composition from the right.

\begin{lemma}
    The action of $\mathcal G$ on $(\mathcal M,\Omega)$ is symplectic.
\end{lemma}

Indeed, for any $\varphi\in \mathcal G$ one has
    \begin{multline*}
        \Omega_{f\circ\varphi}(u\circ\varphi,v\circ\varphi) = 
        \int_S \omega(u\circ\varphi,v\circ\varphi)\sigma = \\ = \int_S \omega(u\circ\varphi,v\circ\varphi)\varphi^*\sigma
        = \int_S \varphi_*(\omega(u,v)\sigma) = \Omega_f(u,v).
    \end{multline*}
\smallskip

Now let $\mathfrak{g} = \vect_{\sigma}(S)$ denote the Lie algebra of the Lie group $\mathcal G$.
Recall that its (smooth) dual is $\Omega^1(S)/d\Omega^0(S)$, where $\Omega^k(S)$ is the space of smooth $k$-forms on $S$.
We now recall the momentum map (following Donaldson \cite{Do1999}).

\begin{definition}\label{def:momentum_map}
    A \emph{momentum map} for the action of $\mathcal G$ on $\mathcal M$ is a mapping $\mu\colon \mathcal M \to \mathfrak{g}^*$ such that for every $\xi\in \mathfrak{g}$
    \begin{equation*}
        d\langle \mu(\cdot),\xi\rangle = \iota_{X_\xi}\Omega
    \end{equation*}
    where $X_\xi$ is the vector field corresponding to the infinitesimal action of $\xi$ on $\mathcal M$.
    A momentum map $\mu$ is called \emph{equivariant} if for every $f\in\mathcal M$ and every $\varphi\in \mathcal G$
    \begin{equation*}
        \mu(f\circ\varphi) = \operatorname{Ad}_\varphi^*\mu(f).
    \end{equation*}
\end{definition}

\begin{proposition}[Donaldson~\cite{Do1999}]
    If $S$ has trivial first cohomology, and, for $f\in \mathcal M$, the pullback 2-form $f^*\omega$ is in the vanishing second cohomology class of $S$, then a momentum map for the volume-preserving diffeomorphism action is given by $\mu(f) = [a]\in \Omega^1(S)/d\Omega^0(S)=\mathfrak X^*_{\sigma}(S)$, where $f^*\omega = d a$. 
\end{proposition}

In what follows we will always assume that $S$ has zero first cohomology to make simpler the consideration of various symplectic structures on its submanifolds. 


\subsection{Zero-momentum symplectic reduction}

Recall the zero-momentum reduction
\begin{equation*}
    \mathcal M/\!/\mathcal G = \mu^{-1}(0)/\mathcal G.
\end{equation*}

    In the case of embeddings $S\to M$, the space $\mathcal M/\!/\mathcal G$ can be identified with the space
    of isotropic weighted submanifolds in $M$, see Donaldson~\cite{Do1999}.
    Indeed, the zero momentum implies that the symplectic structure vanishes on the pre-images, which stands for the isotropic property.

\begin{example}\label{ex:zero-dim}
    Let $S$ be zero-dimensional with $N$ components and density corresponding to strength coefficients $\kappa_1,\ldots,\kappa_N\in \mathbb{R}$.
    Then $\mathcal M \simeq M^N$ and the symplectic structure is
    \begin{equation*}
        \Omega\big((u_1,\ldots,u_N),(v_1,\ldots,v_N) \big) = \sum_{i=1}^N \kappa_i \omega(u_i,v_i).
    \end{equation*}
    The identity component of $\mathcal G$ is just one point, so its Lie algebra is $\{0\}$.
    Thus, $\mathcal M/\!/\mathcal G = \mathcal M$ in this case.
\end{example}

\begin{example}
    Let $2\operatorname{dim}S \leq \operatorname{dim}M$ and consider the open subset of embeddings $\mathrm{Emb}(S,M) \subset \mathcal M$.
    Then $f\in\mu^{-1}([0])$ implies that the submanifold $\Gamma = f(S)\subset M$ is isotropic, i.e., $f^*\omega = 0$.
    Two elements $f,f'\in \mu^{-1}([0])$ belong to the same coset in $\mu^{-1}(0)/\mathcal G$ if and only if $f' = f \circ \varphi$ for some $\varphi\in \mathrm{Diff}_\sigma(S)$.
    Since $f$ is a diffeomorphism between $S$ and $\Gamma$, we obtain a volume form $\varrho = f_*\sigma$ on $\Gamma$. 
    Thus, the space $\mathcal M/\!/\mathcal G$ is identified with 
    the submanifold $\Gamma$ as a base with fiber $\mathrm{Diff}(\Gamma)/\mathrm{Diff}_\varrho(\Gamma) = \mathrm{Dens}(\Gamma)$.
    Thus, from the symplectic structure on the  space $\mathcal M$ we obtain a symplectic structure on the fiber bundle $\mathcal E$ of volume forms sitting over isotropic submanifolds in the symplectic manifold $M$ (Weinstein~\cite{We1990} calls such a pair $(\Gamma, \varrho)\in \mathcal E$ a \emph{weighted isotropic submanifold}). We discuss this symplectic structure in the next section.
\end{example}

\begin{example}
    At the other extreme, consider the case when $\operatorname{dim} S > \operatorname{dim}M$ and we take the open subset of submersions $\mathrm{Sub}(S,M)\subset \mathcal M$.
    Then $f \in \mu^{-1}([0])$ implies that $f^*\omega = 0$ which cannot happen if $f$ is a submersion.
    Thus, to obtain a nonempty set in the preimage of $\mu$ one needs to consider the reduction at general momentum $[a]\in \mathfrak{g}^*$, namely
    \begin{equation*}
        \mu^{-1}([a])/\mathcal G_{[a]} \simeq \mu^{-1}(\operatorname{Ad}_{\mathcal G}^*[a])/\mathcal G\quad \text{where}\quad\mathcal G_{[a]} = \{\varphi\in \mathcal G\mid \varphi^* da = da \}.
    \end{equation*}
    The space $\mu^{-1}([a])$ consists of all $f$ such that $f^*\omega = da$.
\end{example}

\begin{remark}
 The setting of the symplectic structure on the space $\mathcal M$ of smooth maps $S\to M$
 is the framework of the dual pair: one has the actions of volume preserving diffeomorphisms $\mathrm{Diff}_\sigma(S)$ 
 on the right and Hamiltonian diffeomorphisms $\Ham(M)$ on the left, and hence one consider two complimentary momentum maps
 \begin{equation}\label{dupa}
     \ham(M)^* \longleftarrow \mathcal{M} \longrightarrow \mathrm{vect}_\sigma(S)^*.
 \end{equation}
 The corresponding dual pairs were discussed by Marsden and Weinstein~\cite{MW1983} and by Gay-Balmaz and Vizman~\cite{GBV12}.  
\end{remark}


\section{Weinstein's setting}\label{sec:Weinstein}

Donaldson's setting from the previous section, restricted to the special case when the space of mappings $S\to M$ are embeddings, provides a natural framework for the following symplectic structure defined by Weinstein~\cite{We1990}.

Let $(M,\omega)$ be a symplectic manifold as before and let $(\Gamma,\rho)$ a weighted isotropic submanifold in $M$.
In other words, $\Gamma\subset M$ is a compact oriented $k$-dimensional isotropic submanifold and $\rho$ is a  $k$-form on $\Gamma$, with $k\le \dim M/2$.
Now regard $(\Gamma,\rho)$ as a linear functional on the Lie algebra of Hamiltonian functions on $M$.
Namely, let $\ham(M)=C^\infty(M)$ be the Lie algebra of Hamiltonians with respect 
to the Poisson bracket on $M$.\footnote{For a noncompact $M$, the Hamiltonians are assumed to be fast-decaying at infinity.}
The (smooth) dual space is $\ham^*(M)=\Omega^{2n}(M)$ and the weighted submanifold $(\Gamma, \rho)$ gives the singular linear functional on $\ham(M)$
$$\langle (\Gamma, \rho), h\rangle:=\int_\Gamma h\,\rho\,,
$$
where $h$ is a  Hamiltonian function
(of a Hamiltonian vector field $v_h$ on $M$). Note that one can regard $(\Gamma, \rho)$ as a distribution $\rho\wedge\delta_\Gamma$ on $M$, so that the pairing becomes natural. Thus, $(\Gamma, \rho)$ belongs to the closure $\overline{\ham}^*(M)$, in the sense of distributions, of the smooth dual space $\ham^*(M)$. 
Below we assume that $\rho$ is nowhere zero, i.e., it is a volume form on $\Gamma$.
\medskip

Consider the coadjoint orbit $\mathcal O_{(\Gamma, \rho)}$ of $(\Gamma, \rho)\in \ham^*(M)$. There is a simple description of a
neighborhood of $(\Gamma, \rho)$ provided that $\Gamma\subset (M, \omega) $ is isotropic, i.e. $\omega|_\Gamma=0$. Indeed, in this case, by Darboux--Weinstein theorem, a neighborhood of $\Gamma$ in $M$ is a direct product, and is symplectomorphic to the neighborhood in the normal bundle of $\Gamma\subset M$,  while all $C^1$-close isotropic submanifolds in a neighborhood of $\Gamma$ can be obtained by a symplectic deformation of $\Gamma$.
Consider the space ${\mathcal Is}$ of all weighted {\it isodrastic submanifolds}  $(\Gamma', \rho')$ near $\Gamma$, i.e., those submanifolds $\Gamma'$ that are obtained from isotropic submanifold $\Gamma$ by a $C^1$-small Hamiltonian isotopy,  and with $\rho'$ $C^1$-close to $\rho$ and whose total volumes are equal: $\int_{\Gamma'}\rho'=\int_\Gamma\rho$. (As described in \cite{We1990}, 
isodrastic submanifolds can be characterized by constancy of the corresponding actions integrals.)

\begin{proposition}[\cite{We1990,Lee_2009}]\label{p31}
    The Weinstein symplectic structure on the space ${\mathcal Is}$ near $(\Gamma, \rho)$ coincides 
with the Kirillov--Kostant symplectic structure on the coadjoint orbit $\mathcal O_{(\Gamma, \rho)}\subset \overline{\ham}^*(M)$.
Namely, for two Hamiltonian functions $H_\xi$ and $H_\eta$ for the Hamiltonian vector fields $\xi$ and $\eta$ on $(M, \omega)$ defined in a neighborhood of $\Gamma$, 
\begin{equation}\label{wein}
    \Omega^W_{(\Gamma, \rho)}(\xi, \eta) = \int_{\Gamma} \{H_\xi, H_\eta\} \,\rho=\int_{\Gamma} \omega(\xi, \eta) \,\rho\,.
\end{equation}
\end{proposition}

Note that the bracket at $ (\Gamma, \rho)$ is defined by the values of the vector fields $\xi$ and $ \eta$ on $\Gamma$ itself, rather than in a neighborhood of $\Gamma$.
The symplectic structure is defined on the space of isotropic submanifiolds in the vicinity of $\Gamma$ of the same dimension $k$, where  $0\le k\le n=\dim M/2$. 
While the bracket can be defined by the same formula on not necessarily isotropic submanifolds of any dimension, the set of the nearby submanifolds  obtained from $\Gamma$ by a Hamiltonian isotopy is more difficult to describe, as the existence of a Hamiltonian deformation implies that nearby submanifolds must have symplectomorphic restrictions to them  of the symplectic structure $\omega$. 

We mention that the infinite-dimensional geometric setting of convenient calculus was developed for Weinstein's symplectic structure by Lee~\cite{Lee_2009}.


\section{Marsden--Weinstein symplectic structures}\label{sec:MW}

\subsection{Membranes (codimension 2 submanifolds)}
Recall the definition of the Marsden--Weinstein symplectic structure on the space $\mathcal{C}$ of codimension 2 oriented compact submanifolds of a manifold with a volume form.
Thus,
$\Sigma\in \mathcal{C}$ is an oriented, 0-homologous, codimension 2 submanifold\footnote{There are versions of this theory for more general submanifolds. For instance, for non-zero-homologous submanifolds one has to pass to central extensions of the diffeomorphism groups.}
 (also called a membrane) of a manifold $S$ equipped with a volume form~$\sigma$. 
Now regard~$\Sigma$ as a linear functional (singular in a sense of distributions) on the space $\vect_\sigma(S)$ of exact divergence-free vector fields in $S$. Namely, 
given $v\in \vect_\sigma(S)$, the pairing is 
\begin{equation}\label{flux}
\langle \Sigma, v\rangle = {\rm Flux}\, v|_\Sigma = \int_{\partial^{-1}\Sigma} i_v \sigma\,,
\end{equation}
where $\partial^{-1}\Sigma\subset S$ stands for any oriented submanifold whose (compatibly oriented) boundary is $\Sigma$.
The pairing does not depend on the choice of this submanifold. 
That is, $\Sigma\in \overline{\vect}^*_\sigma(S)$, where $\overline{\vect}^*_\sigma(S)$ is the space of currents given by the completion of the corresponding smooth dual $\vect^*_\sigma(S)$.

\begin{definition}[\cite{MW1983}] The {\it Marsden--Weinstein symplectic structure} on $\mathcal{C}$ is defined by the formula
$$
    \Omega^{MW}_\Sigma(U, V) := \int_\Sigma i_Ui_V\,\sigma
$$
where two tangent vectors $U,V$  are  regarded as a pair of variations of the membrane $\Sigma\subset S$, i.e., two vector fields in $S$ defined only on $\Sigma$.
If the manifold $S$ is of dimension $n$, the submanifold $\Sigma$ is of dimension $n-2$, over which one integrates the $(n-2)$-form $i_Ui_V\,\sigma$ defined on $\Sigma$.
\end{definition}

For the fields $U$ and $V$ defined on $\Sigma\subset S$, consider now an arbitrary extension to divergence-free vector fields $u$ and $v$ on $(S, \sigma)$. 

\begin{proposition}\label{prop:MWvsKK}
The Marsden--Weinstein symplectic structure on $\mathcal{C}$ coincides with the Kirillov--Kostant symplectic structure 
on the coadjoint orbit $\mathcal O_\Sigma\subset \overline{\vect}^*_\sigma(S)$.
\end{proposition}

The corresponding Kirillov--Kostant symplectic structure on orbits in $\overline{\vect}^*_\sigma(S)$
is given as follows. 
Regard the extension fields $u$ and $v$ in $S$ as variations of the vorticity 2-form $\delta_\Sigma\in \overline{\vect}^*_\sigma(S)$. 
The Kirillov--Kostant symplectic structure on a coadjoint orbit $\mathcal{O}_\Sigma$ associates to a pair of such variations, tangent to the coadjoint orbit of the vorticity $\delta_\Sigma$, the following quantity:
$$
 \Omega^{KK}_\Sigma(u,v) := \int_S (i_{[u,v]} \zeta)\, \sigma \,,
$$   
where the 1-form $\zeta :=d^{-1}\delta_\Sigma$ is any primitive of the vorticity 2-form $\delta_\Sigma$ supported on $\Sigma$, while $[u,v]$ is the commutator of the vector fields $u$ and $v$ on $S$. 

\begin{proof}[Proof of Proposition~\ref{prop:MWvsKK}]
The proof reduces to the following sequence of identities.
For an exact vorticity 2-form $\xi=d\eta$ one has
$$
 \Omega^{KK}_\xi(u,v) := \langle d^{-1}\xi, [u,v]\rangle=\langle \eta, [u,v]\rangle=
 \int_S \eta\wedge i_{[u,v]} \sigma = \int_S \xi\wedge i_u i_v\sigma \,,
$$
Here we used that for divergence-free vector fields $u$ and $v$ their commutator satisfies the identity 
$i_{[u,v]}\sigma = d(i_u  i_v \sigma)$, which implies the last equality. 
For a singular $\xi=\delta_\Sigma$ the final integral reduces to $ \int_\Sigma i_Ui_V\,\sigma$, which is the definition of $\Omega^{MW}_\Sigma(U, V)$.
We refer to \cite{ArnoldKh, HV2004, Kh2012} for more details.
\end{proof}

\begin{remark}\label{rem:MW-filaments}
One can also describe the Marsden--Weinstein symplectic on membranes
within the Donaldson framework applied to submersions $S \to M$. 
Namely, let $\mathcal M$ be  the space of proper submersions $(S,\sigma)\to (M, \omega)$ with $\operatorname{dim} S > \operatorname{dim}M=2$. Consider first 
the action on $\mathcal M$ of the symplectomorphism group $\Symp_\omega(M)$ from the left. 
It is  symplectic on the space of maps $\mathcal M$ and 
transitive on $M$. 
Take any point $q\in M$ and consider the symplectic form $\omega_q$ on the tangent space $T_qM$. 
(Here $q\in M$ can be regarded as a zero-dimensional weighted isotropic submanifold of weight 1, while one considers the preimage of a point orbit for the momentum map $\mathcal M\to \symp^*_\omega(M)$ for the $\Symp_\omega(M)$-action.)
By considering the space of pre-images $\Sigma=f^{-1}(q)\subset S$ for various submersions $f$ with a noncritical value $q$
we obtain the space of codimension 2 closed cooriented submanifolds in $S$. 
Finally, via quotients by the action of the group $\mathcal G = \mathrm{Diff}_\sigma(S)$ on those pre-images from the right we obtain the space of codimension 2 closed (co)oriented submanifolds $\Sigma$ in  $S$ considered modulo diffeomorphism action. 

One can see that this is the Marsden--Weinstein symplectic structure on the space of vortex membranes, as the latter can be regarded as the averaged  symplectic structure in ``normal directions" to $f^{-1}(q)\subset S$. This is the Hamiltonian reduction from the symplectic structure on $\mathcal M$
first by {\it restriction} to all pre-images of a point, and then by taking the {\it quotient} by the $\mathrm{Diff}_\sigma(S)$-action, see Figure~\ref{fig}. 
For $\operatorname{dim} S=3$ and $\operatorname{dim}M=2$ we obtain the space of non-parametrized 
knots in $S$ equipped with the Marsden--Weinstein symplectic structure. We discuss this relation in more detail in Appendix~\ref{sect:reduction}.
\end{remark}

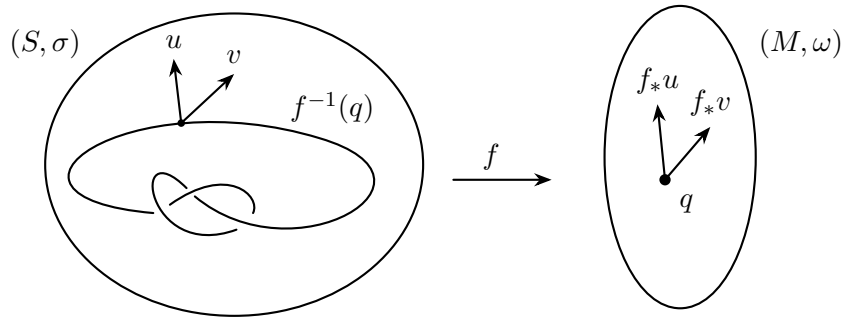
\begin{figure}[h]
\centering
\begin{tikzpicture}[scale=1, >=Stealth, line cap=round, line join=round]

\draw[thick] (-2.7,.2) ellipse (2.5 and 2);
\node at (-5.2,1.8) {$(S,\sigma)$};
\node at (-1.4,.95) {$f^{-1}(q)$};

\begin{scope}[shift={(-3,0)}, scale=1.1]

\draw[thick]
(-.7,-.4)
.. controls (-3,-.2) and (-1,1.2)  .. (1.3,0.5)
.. controls (3, 0) and (1, -1.2) .. (-.2, -.2);

\draw[thick]
(-0.3,-.1)
.. controls (-1,0.6) and (-.8,-1) .. (0.3, -0.6);

\draw[thick]
(0.5,-.4)
.. controls (.6,-.2) and (.2,.2) .. (-0.5, -0.3);

\end{scope}

\draw[->, thick] (-3.4,0.75) -- (-3.5,1.6) node[above] {$u$};
\draw[->, thick] (-3.4,0.75) -- (-2.7,1.4) node[above] {$v$};
\node[circle, fill, inner sep=1pt, ] at (-3.4,0.75) {};

\draw[->, thick] (0.2,0) -- (1.5,0);
\node at (0.7,0.3) {$f$};

\draw[thick] (3.2,0.3) ellipse (1 and 2);
\node at (4.8,1.8) {$(M,\omega)$};

\draw[->, thick] (3,0) -- (2.9,1.0) node[above] {$f_*u$};
\draw[->, thick] (3,0) -- (3.6,0.7) node[above] {$f_*v$};
\node[circle, fill, inner sep=1.5pt, label=below right:$q$] at (3,0) {};

\end{tikzpicture}
\caption{The Marsden--Weinstein symplectic structure on the space of knots can be regarded as the averaged  symplectic structure in ``normal directions" to the fibers of $f:S\to M$.}\label{fig}
\end{figure}

\subsection{Vortex sheets (codimension 1 submanifolds)}

It turns out that a similar symplectic structure also exists on the space of vortex sheets $(P, \beta)$ in $S$. 
Here, a vortex sheet $(P, \beta)$ consists of a hypersurface $P\subset S$ equipped with a 1-form $\beta$, so that the current $\beta\wedge \delta_P$ is an exact 2-form. (Note that $\delta_P$ is a singular 1-form (current) with support on $P$.)

\begin{definition}[\cite{Kh2012}]\label{def:vs}
The {\it symplectic structure} on the space of vortex sheets $\mathcal{VS}$ is defined by the formula
$$
    \Omega^{VS}_{(P, \beta)}(V, W) := \int_P \beta\wedge (i_Vi_W\,\sigma)
$$
for a pair of divergence-free vector fields $V$ and $W$ in $\mathcal{VS}$ that are  regarded as a pair of variations of the vortex sheet $(P, \beta)$.
Here if the manifold $S$ is of dimension $n$, the hypersurface $P$ is of dimension $n-1$, over which one integrates the wedge product of the 1-form $\beta$ and
the $(n-2)$-form $i_Vi_W\,\sigma$ defined on $P$.
\end{definition}

\begin{remark} More precisely, the space $\mathcal{VS}$ associated with a vortex sheet $(P, \beta)$ contains all pairs 
$(\tilde P, \tilde \beta)$ diffeomorphic to it by means of volume-preserving diffeomorphisms of $S$, i.e. $\mathcal{VS}$  is a volume analogue of the space of isodrastic submanifolds, cf. \cite{HV2025}. 
\end{remark}

Consider arbitrary extensions of the fields $V$ and $W$ defined on $P\subset S$ to divergence-free
vector fields $v$ and $w$ on $(S, \sigma)$. A  vortex sheet $(P, \beta)$ can be regarded as a functional on the space of exact divergence-free vector fields in $(S, \sigma)$: given $v\in \vect_\sigma(S)$, the pairing is 
$$
\langle (P, \beta), v\rangle = \int_P i_v d^{-1}(\beta\wedge\delta_P)\,\sigma\,,
$$
where $d^{-1}(\beta\wedge\delta_P)$ is a primitive 1-form for the vorticity 2-form $\beta\wedge\delta_P$. The pairing is
well-defined, i.e., it does not depend on the choice of primitive.

\begin{proposition}[\cite{Kh2012}]
The Marsden--Weinstein symplectic structure on $\mathcal{VS}$ coincides with the Kirillov--Kostant symplectic structure 
on the coadjoint orbit $\mathcal O_{(P, \beta)}\in \vect^*_\sigma(S)$.
\end{proposition}

While vorticity is always an exact 2-form (or 2-current), one can extend the Marsden--Weinstein symplectic structure to the case of nonexact $\beta\wedge \delta_P$, see \cite{GBV2024}.

\begin{remark}\label{rem:MW-sheets}
Similarly to the case of membranes, one can put the symplectic structure on vortex sheets within Donaldson's framework. 
Namely, take a weighted 1-dimensional submanifold $(L, \lambda)$ in the symplectic surface $(M, \omega)$, where $L\subset M$ is a closed curve, 
$\lambda$ is a 1-form on $L$, and $\operatorname{dim}M=2$. 
Consider now the space of pre-images $(f^{-1}(L), f^*\lambda)$ in $S$ for various submersions $f$. This time they 
can be regarded as the space of vortex sheets $(P, \beta)$: codimension 1 closed oriented submanifolds $P=f^{-1}(L)$ in $(S, \sigma)$ equipped with a closed 1-form $\beta=f^*\lambda$. (More precisely, for vortex sheets the pair is supposed to represent an exact form, e.g., provided if  the preimage $f^{-1}(L)\subset S$ is null-homologous.) The corresponding Hamiltonian reduction from the symplectic structure on weighted 1-dimensional submanifolds in $(M, \omega)$ induces an analog of the Marsden--Weinstein symplectic structure on vortex sheets in $S$, see \cite{Kh2012}.  We refer to  Appendix~\ref{sect:reduction} for a  more formal exposition. 
\end{remark}


\section{Symplectic structure on singular 1-form densities}\label{sec:triples}

In this section we introduce a new symplectic structure motivated by the singular vorticities in compressible fluids and 
modeled by augmented triples or singular 1-form densities, as  in \cite{HV2025}.
It can be viewed as an extension of the Marsden–Weinstein structure, from the incompressible to the compressible case.

To describe a similar problem for the group ${\rm Diff}(Q)$ of {\it all diffeomorphisms} of a manifold $Q$, consider its Lie algebra $\vect(Q)$ of all vector fields and its smooth dual space $\vect^*(Q)$. 
We are going to define an analog of the Marsden--Weinstein symplectic structure on singular elements of the corresponding full dual $\overline{\vect}^*(Q)$. 
In a sense, this structure occupies an intermediate position between the Weinstein symplectic structure on weighted isotropic submanifolds ${\mathcal Is}$ and the Marsden--Weinstein symplectic structure on membranes $\mathcal{C}$.
\medskip

First, note that the smooth dual $\vect^*(Q)$ of the  Lie algebra $\vect(Q)$  
can be identified with the space $\Omega^1(Q)\otimes \Omega^n(Q)$, where $Q$ is of dimension $n$ and the tensor product $\otimes$ is over $C^\infty(Q)$.

\begin{definition}[\cite{HV2025}]
    Let $(\Gamma, \alpha\otimes\rho)$ be a {\it  singular 1-form density} where $\Gamma\subset Q$ is a compact oriented
    $k$-dimensional submanifold of $Q$, 
    $\rho$ is a (volume) $k$-form on $\Gamma$, and $\alpha\in \Omega^1_\Gamma(Q)$ is a 1-form on $Q$ defined only at the points of $\Gamma$.
    It can be regarded as a singular linear functional on the Lie algebra $\vect(Q)$ as follows:
    \begin{equation}\label{incl}
    \langle (\Gamma, \alpha\otimes\rho), v\rangle:=\int_\Gamma (i_v\alpha)\rho,
    \qquad\text{for}\quad v\in \vect(Q).
    \end{equation}
    Thus, we think of the 1-form density as a current $\alpha\otimes\rho\wedge \delta_\Gamma \in \overline{\vect}^*(Q)$, where the tensor product $\otimes$ is over $C^\infty(\Gamma)$. 
\end{definition}

Diffeomorphisms of $Q$ act on the 1-form densities in a natural way: $\varphi\in \Diff(Q)$ acts by $\varphi^*(\alpha\otimes\rho)\wedge \delta_{\varphi^{-1}(\Gamma)}$.  
We denote by $\Aug$ the space of singular 1-form densities with $k$-dimensional support in $Q$ of a given diffeomorphism type.
One can also associate an {\it augmented triple} $(\Gamma, \rho, \alpha)$ to $\alpha\otimes\rho\wedge \delta_\Gamma$. Such a triple is not unique: triples $(\Gamma, \rho, \alpha)$ and $(\Gamma, \rho/f, f\alpha)$
for a nonvanishing $f\in  C^\infty(\Gamma)$ correspond to the same 1-form density. Once a decomposition of $\alpha\otimes\rho\wedge \delta_\Gamma$ into the triple $(\Gamma, \rho, \alpha)$ is fixed, a diffeomorphism $\varphi\in \Diff(Q)$ acts naturally by $(\varphi^{-1}(\Gamma), \varphi^*\rho, \varphi^*\alpha)$.
\medskip

Below we  define a natural symplectic structure on the space $\Aug$  of singular 1-form densities, as well as 
the corresponding presymplectic structure on the space of augmented triples. 

\medskip

\begin{theorem} \label{thm:triple_ss}
    The generalized Marsden--Weinstein symplectic structure on
the space $\Aug$ the space of singular 1-form densities $\alpha\otimes\rho\wedge \delta_\Gamma$ is defined by formula 
    \eqref{eq:triple} below and coincides with   the Kirillov--Kostant symplectic structure on the coadjoint orbit $\mathcal O_{(\Gamma, \alpha\otimes\rho)}\subset \overline{\vect}^*(Q)$. The pullback of the corresponding symplectic form defines a presymplectic structure on the space $\mathcal Tri$ of augmented triples $(\Gamma, \rho, \alpha)$.
\end{theorem}

\begin{remark}
The space of {\it nowhere vanishing} singular 1-form densities is called {\it augmented non-linear Grassmannian}. 
Their coadjoint orbits and other properties were studied in detail from the point of view of symplectic reduction 
and EPDiff dual pairs in  \cite{HV2025}.
Below we extend this consideration to the case of arbitrary  1-form densities and give an explicit formula
for the corresponding symplectic structure. We also single out the case where the 1-form $\alpha$ can be chosen to have closed restriction $\alpha|_{\Gamma}$, i.e. $d(\alpha|_{\Gamma})=0$: the symplectic structure on $\Aug$  turns out to be more natural in that case.
\end{remark}

\begin{proof}
First we recall that any diffeomorphism of a manifold $Q$ gives rise to a symplectomorphism of $T^*Q$
with the standard symplectic structure. In particular, the vector field $v(q)\partial/\partial q$ on $Q$ lifts to
the Hamiltonian field $v(q)\partial/\partial q-p\cdot Dv(q)\partial/\partial p$ with the Hamiltonian
$H_{\tilde v}=p\cdot v(q)$ on $T^*Q$ (the evaluation map at $v$).

Next we decompose  elements $(\Gamma, \alpha\otimes\rho)$ of $\Aug$ into triples $(\Gamma, \rho, \alpha)$, 
with $\alpha\in \Omega^1_\Gamma(Q)$ and $\rho\in\mathrm{Vol}(\Gamma)$, and fix this decomposition for now.
Each triple  determines a weighted submanifold in $T^*Q$. 
Namely, use the 1-form $\alpha$ to lift $\Gamma\subset Q$ to $\widetilde \Gamma\subset T^*Q$, 
since $\alpha $ defines a section of the bundle $\pi: T^*Q\to Q$.  Moreover, for $\rho\in\Omega^k(\Gamma)$ we define $\tilde\rho:=\pi^*\rho\in \Omega^k(\widetilde\Gamma)$. Thus, $(\widetilde \Gamma, \tilde \rho)$ is 
a weighted submanifold in $T^*Q$.


Note that one obtains an {\it isotropic weighted submanifold} if and only if $d(\alpha|_\Gamma)=0$.

\begin{lemma}
The restriction $\alpha|_\Gamma$ is closed on $\Gamma$ if and only if the lift $\widetilde \Gamma\subset T^*Q$ is isotropic.
\end{lemma} 

\begin{proof}
The 1-form $\alpha$ defining the lift of $\Gamma$ to $T^*Q$ is  the restriction of the Liouville 1-form $\theta$ on $ T^*Q$ to $\widetilde \Gamma\subset T^*Q$. The restriction commutes with taking the differential. Then $\alpha|_\Gamma$ is closed on $\Gamma$, 
  if and only if  $\omega|_{\widetilde \Gamma} =d\theta |_{\widetilde \Gamma} = d\alpha|_\Gamma=0$, i.e. $\widetilde \Gamma$ is isotropic.
\end{proof}

Furthermore, the weighted submanifold  $(\widetilde \Gamma, \tilde \rho)$ can be regarded as a linear functional on $\ham(T^*Q)$ and hence an element of $\overline{\ham}^*(T^*Q)$, where the pairing with a Hamiltonian field $\tilde v$ with Hamiltonian $H_{\tilde v}$ on $T^*Q$ is
$$
\langle (\widetilde \Gamma, \tilde \rho), \tilde v\rangle:= \int_{\tilde \Gamma} H_{\tilde v}\, \tilde \rho\,.
$$
In this way we get a $\Diff(Q)$-equivariant injection
\[
\Tri\hookrightarrow \overline\ham(T^*Q)^*,
\]

\begin{lemma}
The linear functional on $\ham(T^*Q)$   corresponding to  $(\widetilde \Gamma, \tilde \rho)$
is well-defined on Hamiltonian fields on $T^*Q$ corresponding to vector fields on~$Q$, i.e., it does not depend on the 
splitting of $\alpha\otimes \rho$ into the pair of 1-form $\alpha$ and volume form $\rho$ on $\Gamma$.
\end{lemma} 

\begin{proof}
The tensor product $\alpha\otimes \rho$ is over $C^\infty(\Gamma)$, so consider a different splitting
$\alpha'\otimes \rho'$ where $\alpha'=f\alpha$ and $ \rho'= \rho/f$ for $f\in C^\infty(\Gamma)$.
For a field $v$ on $Q$, lifted to a Hamiltonian field $\tilde v$ on $T^*Q$, 
 Then
$$
\int_{\tilde \Gamma} H_{\tilde v}\, \tilde \rho=
\int_{\tilde \Gamma}\tilde  \alpha(\tilde v)\, \tilde \rho=
\int_{ \Gamma}   \alpha(  v)\,  \rho=
\int_{ \Gamma}   f\alpha(  v)\,  (\rho/f)=
\int_{ \Gamma}   \alpha'(  v)\,  \rho'\,,
$$
where we used that by definition of the lift, $\tilde \alpha(\tilde v)|_{\tilde \Gamma}=H_{\tilde v}|_{\tilde \Gamma}$, as both sides at a point $(q,p)$ are equal to $p\cdot v(q)$. One can see that the pairing  is well-defined, since it involves only the product $\alpha\otimes \rho$.
\end{proof}

Now we can define the Kirillov--Kostant symplectic structure on the spaces of diffeomorphic 1-form densities $\Aug$ or
equivalence classes of triples $\mathcal Tri$. Recall that the non-degenerate pairing \eqref{incl} defines a $\Diff(Q)$-equivariant injection 
\[
\Aug\hookrightarrow \overline{\vect}^*(Q).
\]
The $\Diff(Q)$ coadjoint orbits in $\Aug$
are endowed with orbit symplectic forms
\begin{equation}\label{omkk}
 \Omega^{KK}_{(\Gamma, \alpha\otimes\rho)}(u,v)= \langle (\Gamma, \alpha\otimes\rho), [u,v] \rangle 
  = \int_\Gamma  \alpha([u,v])\,\rho,
\end{equation}
for vector fields $u,v \in \vect(Q)$, see e.g. \cite{HV2025}.

Now represent a 1-form density $(\Gamma, \alpha\otimes\rho)$  in $Q$ of any dimension $k$, $0\le k\le n=\dim Q$ by  an augmented triple $(\Gamma, \rho, \alpha)$ and consider two vector fields $u,v \in \vect(Q)$. 
They can be regarded as infinitesimal variations of the  triple.
Lift the augmented triple $(\Gamma, \rho, \alpha)$ in $Q$ to the weighted submanifold $(\tilde\Gamma, \tilde\rho)$ in $T^*Q$ of the same dimension $k$. 
The vector fields $u$ and $v$ are lifted to Hamiltonian fields $\tilde u$ and $\tilde v$ in $T^*Q$ with Hamiltonian functions $H_{\tilde u}(q,p) = \langle p, u(q)\rangle$ and $H_{\tilde v} = \langle p, v(q)\rangle$.
Thus, two such variations are given by Hamiltonian fields in $T^*Q$, whose Hamiltonians, linear in momenta $p$, are $H_{\tilde u}$ and $H_{\tilde v}$. 
Their Poisson bracket  $\{H_{\tilde u}, H_{\tilde v}\}$ restricted to and integrated over $\widetilde \Gamma$ against the density $\tilde \rho$ is the following closed 2-form on those triples:
 $$
 \Omega_{(\Gamma, \rho, \alpha)}(u,v)
=\int_{\tilde \Gamma} \{H_{\tilde u}, H_{\tilde v}\}\tilde \rho
= \int_{\tilde \Gamma} \omega_{T^*Q}({\tilde u}, {\tilde v})\tilde \rho\,.   
$$ 
Here, $\omega_{T^*Q}$ is the canonical symplectic structure on $T^*Q$. 

On the other hand, integration of the  Poisson bracket  $\{H_{\tilde u}, H_{\tilde v}\}$ restricted to $(\widetilde \Gamma, \tilde \rho)$ gives exactly  the Kirillov--Kostant symplectic structure on the coadjoint orbits of weighted submanifolds in $T^*Q$. 
Indeed, by definition of the Kirillov--Kostant bracket on $\mathcal O_{(\Gamma, \alpha \otimes\rho)}\subset \overline{\vect}^*(Q)$
is 
$$
 \Omega^{KK}_{(\Gamma, \alpha \otimes\rho)}(u,v) \coloneqq  \int_\Gamma  \alpha([u,v])\,\rho 
    = \int_{\tilde \Gamma} \tilde  \alpha(\widetilde{[u,v]})\,\tilde\rho
 =\int_{\tilde \Gamma} \{H_{\tilde u}, H_{\tilde v}\}\tilde \rho\,,
$$
where the last two identities hold due to the Lie algebra homomorphism  $\sim \colon \vect(Q)\to \ham(T^*Q)$, with 
$ H_{\widetilde{[u,v]}}=\{H_{\tilde u}, H_{\tilde v}\}$ and definition of the lift to $T^*Q$.
In particular the latter shows that the independence of the symplectic structure of the decomposition of $(\Gamma, \alpha \otimes\rho)$
into triples $(\Gamma, \rho, \alpha)$, so it is a well-defined symplectic structure on the equivalence classes
$(\Gamma, \rho, \alpha)\simeq (\Gamma,  \rho/f, f\alpha)$. This is natural to call the {\it generalized Marsden--Weinstein symplectic structure} on the singular 1-form densities (or equivalence classes of those triples):
\begin{equation}\label{eq:triple}
 \Omega^{GMW}_{(\Gamma, \alpha \otimes\rho)}(u,v)
=\int_{\tilde \Gamma} \{H_{\tilde u}, H_{\tilde v}\}\tilde \rho\,
\end{equation}
for Hamiltonian lifts  $\tilde u$ and $\tilde v$  to $T^*Q$ of the fields $u$ and $v$ in $Q$.

Note that the pullback of this structure to the space of triples themselves is by definition presymplectic: it is degenerate along the equivalence classes. 
This completes the proof of Theorem~\ref{thm:triple_ss}. 
\end{proof}

Note that, in general, the bracket $\{H_{\tilde u}, H_{\tilde v}\}$  depends on both the fields $u$ and $v$ and   on their first jets, since so do their lifts $\tilde u$ and $\tilde v$ to $T^*Q$. This explains why the generalized Marsden--Weinstein bracket is more complicated than the classical ``pointwise" Marsden--Weinstein structure in the volume-preserving setting.

\begin{remark}
Consider in more details the space of all decompositions of a given 1-form density $(\Gamma, \alpha \otimes\rho)$
into triples $(\Gamma, \rho/f, f\alpha)$. As we mentioned, the corresponding pullback is a presymplectic structure on {\it all the triples} projecting 
to ${\mathcal FD}$. The formal argument  is based on the following

\begin{lemma}
Let $H \subseteq G$ be a Lie subgroup, $\psi: \mathfrak{g}^* \to \mathfrak{h}^*$ the corresponding projection, $\mu\in \mathfrak{g}^*$,
and $\psi(\mu)=\nu$. 
We consider $\psi|_{H\cdot\mu}:H\cdot\mu\to\O_\nu$ and we denote by $i$ the inclusion $H \cdot \mu\subseteq \O_\mu$. 
Then, by pullback of the Kirillov--Kostant symplectic forms  $\omega_\mu$ and $\omega_\nu$, one gets a presymplectic form on $H\cdot\mu$:
\begin{equation}\label{pres}
i^*\omega_\mu=(\psi|_{H\cdot\mu})^* \omega_\nu.
\end{equation}
\end{lemma}

\begin{proof}
It immediately follows from the fact that  the projection $\psi: \mathfrak{g}^* \to \mathfrak{h}^*$ is the momentum map for the $H$-action on $\mathfrak{g}^* $. The momentum map is Poisson and hence the preimage of a symplectic coadjoint orbit $\O_\nu\subset \mathfrak{h}^*$ is presymplectic in $\mathfrak{g}^* $.
\end{proof}

Apply this lemma to $\vect(Q)\hookrightarrow\ham(T^*Q)$ with dual map
\[
\psi: \overline\ham(T^*Q)^*\to\overline\vect(Q)^*.
\]
We obtain $\psi(\widetilde \Gamma, \tilde \rho)=(\Gamma,\alpha\otimes\rho)$ for every decomposition of the tensor product.
Its restriction to the  space $\Tri$, which  is the $\Diff(Q)$-orbit of $(\Gamma,\rho,\alpha)$, is simply the application of the tensor product.
The Kirillov--Kostant symplectic forms in \eqref{pres} are $\Omega^W$ from \eqref{wein} on the coadjoint orbit $\O_{(\widetilde\Gamma,\tilde\rho)}$, 
and $\Omega^{KK}$ from \eqref{omkk} on the coadjoint orbit $\O_{(\Gamma,\alpha\otimes\rho)}$.
Thus, the space $\Tri$ inherits the presymplectic form
\[
\Omega^{GMW}:=i^*\Omega^W=(\psi|_{\Tri})^*\Omega^{KK},
\]
where $i$ denotes the inclusion of $\Tri$ in the coadjoint orbit $\O_{\widetilde \Gamma,\tilde\rho}$.
\end{remark}

\begin{remark}
As we discussed above, the pullback of the 2-form $\Omega^{GMW}$ on triples may have a kernel. 
It is given by the  quotient of the Lie algebra of vector fields  preserving $\alpha\otimes\rho$ and the smaller subalgebra of fields preserving separately $\alpha$ and $\rho$, 
i.e. this kernel is given by vector fields
realizing the equivalence $(\Gamma, \rho, \alpha)\simeq (\Gamma,  \rho/f, f\alpha)$. The ``size" of this quotient is a subtle topological question, related to the {\it topology of  kernel distribution} $\tau :=\text{ ker}\,\alpha|_\Gamma$  for the restriction $\alpha|_\Gamma$. 
This kernel $\tau$ is well defined for the 1-form density $\alpha\otimes\rho\wedge \delta_\Gamma$, i.e. for the 
 equivalence classes of triples. For instance,  the space of  vector fields preserving  the distribution $\tau $ differ depending on whether $\tau$ is maximally nonintegrable on $\Gamma$ or, on the contrary, $\tau$ is Frobenius integrable and tangent to a foliation on $\Gamma$. In particular, the latter happens if the restriction $\alpha|_\Gamma$ is a closed 1-form, the case 
corresponding to an isotropic weighted submanifold $(\widetilde \Gamma, \tilde \rho)$ discussed above. 
\end{remark}

\begin{remark}
By considering the lifting map 
$$
(\Gamma, \rho, \alpha)\mapsto (\tilde\Gamma, \tilde\rho)\,,
$$
for the projection $T^*Q\to Q$, coadjoint orbits  of triples in 
$\overline{\vect}^*(Q)$ give rise to coadjoint orbits of weighted submanifolds in $\overline{\ham}^*(T^*Q)$.
In a sense, the symplectic structure on the space of triples occupies an intermediate position between the Weinstein structure on weighted isotropic submanifolds of any dimension $0\le k\le \dim M/2$ of a symplectic manifold $M$ and the Marsden--Weinstein structure on membranes of codimension 2  in a manifold with a volume form.
\end{remark}

\section{Dynamics on membranes and open questions}\label{sec:binormal}

\subsection{The binormal equation and its analogs}
The most relevant Hamiltonian dynamics on curves in 3D associated with the Marsden--Weinstein symplectic structure is the  binormal equation. To define it we fix the Euclidean metric in $\mathbb R^3$
and define the Hamiltonian function on closed curves $\gamma$ to be their length, $H(\gamma)=length(\gamma)$.
Let $\gamma(s)$ is the arc-length parametrization of the curve, and $\gamma':=\partial \gamma/\partial s$.
The following statement is classic by now, see e.g. \cite{ArnoldKh, Calini}:

\begin{proposition}
    The Hamiltonian equation for the length Hamiltonian function and the Marsden--Weinstein symplectic structure  is the binormal equation
    $\partial_t \gamma =\gamma'\times \gamma''$.
\end{proposition}

It has been generalized in \cite{HV2004, Kh2012} to the motion of codimension 2 closed submanifolds $\Sigma$  in a Riemannian manifold $S$ of dimension $d$ equipped with a volume form $\sigma$: for the Hamiltonian function defined as the $(d-2)$-volume of $\Sigma$,  $H(\Sigma)=vol_{d-2}(\Sigma)$, the Hamiltonian evolution is given by the skew-mean-curvature equation
$$
\partial_t q=J({\bf MC}(q))
$$
for any point $q\in \Sigma$, where $\bf MC$ is the mean curvature vector to $\Sigma $ at $q$ and
$J$ is the operator of rotation by $\pi/2$ in the 2D normal plane $N_q\Sigma$.
This equation is also known to be the LIA (Localized Induction Approximation) ansatz of the Euler equation of an incompressible fluid in $(S,\sigma)$, see \cite{Shashi, Kh2012}. 

\begin{proof}
It is known that the mean curvature vector is the direction of the fastest change 
of the volume of a submanifold, for $H(\Sigma)=vol_{d-2}(\Sigma)$ one has 
${\rm grad} \, H(q)=-{\bf MC}(q)$ for $q\in \Sigma$. The Marsden--Weinstein symplectic structure is the averaging of the symplectic sytucture in the normal bundle to $\Sigma$, and at each point this 
symplectic structure is the rotation by  by $\pi/2$ in the 2D normal plane $N_q\Sigma$, i.e. for the skew-gradient, 
${\rm sgrad} \, H(q)=J({\bf MC}(q))$. (More details can be found in \cite{Kh2012}.) For $d=3$ the skew-mean-curvature equation reduces to the binormal equation above.
\end{proof} 

\begin{remark}
    Recall that vortex sheets $(P,\beta)$ are formed by hypersurfaces $P$ with closed or exact  1-forms $\beta$. Integral submanifolds of $\beta$ define a foliation of $(d-1)$-dimensional $P$ into codimension 1 leaves $\Sigma_\beta$, i.e. into $(d-2)$-dimensional vortex membranes. One can show that the natural Hamiltonian
    $H(P)= \int vol_{d-2}(\Sigma_\beta)\, \beta$ computing the ``vortex sheet's volume" defines the dynamics of the vortex sheet that reduces to
    the skew-mean-curvature evolution of each of its membranes, see \cite{Kh2012}.
\end{remark}

\medskip

Now consider compressible fluid in a manifold $Q$ and the  setting of symplectic structure on 
 augmented triples $(\Gamma, \rho, \alpha)\simeq \alpha\otimes\rho\wedge \delta_\Gamma$. 
 Given a Riemannian metric on $Q$, it turns out that the seemingly natural extension of the above ``volume Hamiltonian" function to this case is not enough to define an evolution. 
 Indeed, the Hamiltonian structure requires the first jet of the Hamiltonian function, or the second jet of the vector field on $\Gamma$, so one needs to define the field not only on $\Gamma$, but  in its infinitesimal neighborhood in $Q$.

Note however a similarity with the incompressible setting of vortex sheets:  
 the assumption of closedness of  the restriction 1-form  $\alpha_\Gamma:=\alpha|_\Gamma$ also implies that $\Gamma$ is foliated by the integral 
 submanifolds of $\alpha_\Gamma$. 
Furthermore, while  $\alpha$ and $\rho$ are defined modulo multiplication by $f\in C^\infty(\Gamma)$
such a function $f$ must be constant on the leaves of $\alpha_\Gamma$, since $f\alpha_\Gamma$ also must be closed (indeed, $0=df\alpha_\Gamma = df\wedge\alpha_\Gamma + fd\alpha_\Gamma = df\wedge\alpha_\Gamma$). 
In particular, this defines the volume form on each leaf modulo a constant factor. This might help 
to introduce an appropriate Hamiltonian with an interesting dynamics.

One of possible choices for a Hamiltonian function would be to mimic the one for barotropic fluids: 
given a Riemannian metric on $Q$, which induces a reference volume $k$-form 
$d^k\mathrm{vol}$ on $\Gamma\subset Q$, define a vector field $v$ by $v^\flat \otimes d^k\mathrm{vol} = \alpha\otimes\rho$, and then set the Hamiltonian to be $H(\Gamma,\rho,\alpha) = \int_\Gamma \langle v, \alpha\rangle \rho$. 

\medskip

  Here are several open questions related to the introduction of symplectic structures on the augmented triples.

1) What could be an interesting (and possible integrable in lower dimensions) analog of the binormal equation for membranes in this compressible setting?

2) Is it possible to obtain such an analog of the binormal equation by means of a reduction or ansatz from 
  the compressible Euler equation on $Q$, similar to the localized induction approximation in the incompressible setting?
  
3) What could be an extension of the above symplectic structures to the semidirect product groups 
${\rm Diff}_\sigma(S)\ltimes C^\infty(S)$ and ${\rm Diff} (Q)\ltimes C^\infty(Q)$.

  4) The dual pairs discussed by Marsden--Weinstein \cite{MW1983} and (implicitly) by Donaldson \cite{Do1999}
imply the Morita equivalence of the corresponding manifolds. What would this equivalence give for the coadjoint orbits of the ``singular elements'' and for their quantization?

\subsection{Low-dimensional cases}
It is worth having a closer look  at several low-dimensional cases. 

\medskip

$a)$ The case of $k=0$ for the 
orbits of zero-dimensional submanifolds was described above, see  Example \ref{ex:zero-dim}. 
Here  $(\Gamma, \rho)=\sum_i a_i \, q_i$ for $q_i\in Q$ with strength $a_i\in \mathbb R$, and after
the lift to the symplectic space $T^*Q$,  it corresponds to the product of several copies 
of $T^*Q$ with the symplectic structure $\sum a_i\, dp_i\wedge dq_i$, cf. \cite{MW1983}.
\medskip

$b)$ The case of $k=1$ in a one-dimensional manifold $Q$, $n=1$, corresponds to a smooth, rather than singular, functional on $\vect\,Q$, and yet it is interesting how it fits the main framework. Let $Q=S^1$ and recall that for the Lie algebra $\mathfrak g=\vect(S^1)$ the regular dual is regarded to be the space of smooth quadratic differentials, $\mathfrak g^*=QD(S^1)$. 
There is the natural Lie-Poisson structure on $\mathfrak g^*=QD(S^1)$.

Now suppose $\Gamma=Q=S^1$ and $\rho$ and $\alpha$ are 1-forms on $\Gamma$. Then the triple $(\Gamma, \rho, \alpha)$ can be identified with the quadratic differential $\alpha\otimes\rho\in QD(S^1)$, and the Lie-Poisson 
structure indices the symplectic structure on the coadjoint orbit of $\alpha\otimes\rho$.
In particular, the structure of zeros of $\rho$ and $\alpha$ will affect the description of Casimirs 
in the vicinity of the coadjoint orbit $\mathcal O_{\alpha\otimes\rho}$.

\medskip

$c)$ In the case of a curve $\Gamma$, $k=1$, in a two-dimensional manifold $Q$, $n=2$, consider a 1-form $\rho$ on the curve $\Gamma$ and a 1-form $\alpha$ on $Q$ defined at $\Gamma$. In this case the 1-form density $(\Gamma, \alpha\otimes\rho)$ (or the
triple $(\Gamma, \rho, \alpha)$) is a linear functional on vector fields 
$v\in \mathfrak g=\vect(Q)$ as
$\langle (\Gamma, \alpha\otimes\rho), v\rangle=\int_\Gamma \rho\wedge \alpha(v)$. 
The corresponding Kirillov--Kostant symplectic structure is (somewhat implicitly) defined by
$\int_\Gamma \rho\wedge \alpha([u,v])$ for the commutator $[u,v]$ of two vector fields $u$ and $v$ on $Q$. This 1D case was discussed in more details in  \cite{HV2025}.

This can be compared with the symplectic structure on vortex sheets $(P,\beta)$ in the incompressible setting on $(S,\sigma)$,
where the vortex sheet (here, a curve $P\subset S$) is equipped with a 1-form (here, the 1-form $\beta$ on $P$), see Definition~\ref{def:vs}.
In that case the pairing with divergence-free vector fields in $(S,\sigma)$ is essentially given by their flux through $P\subset S$, so it  does not need fixing the 1-form $\alpha$ 
but relies on the choice of the volume form $\sigma$ on $S$. 

\medskip

$d)$ For a membrane, a submanifold of codimension 2 in $Q$, we have ${\rm dim }\,Q=n$ and ${\rm dim }\,\Gamma =n-2$,
and $\rho$ is a density $(n-2)$-form on $\Gamma$, while $\alpha$ is, as above, 1-form on $Q$ at $\Gamma$.
The corresponding Kirillov--Kostant symplectic structure is defined by the same formula
$\int_\Gamma \rho\wedge \alpha([u,v])$ for vector fields $u$ and $v$ on $Q$. 

Note that in the compressible setting a submanifold $\Gamma$ can be of any dimension $k$ in $n$-dimensional $Q$, and therefore 
analogues of vortex membranes can occur in any codimension!


\section{Appendix: Symplectic manifold of submersions to a surface}\label{sect:reduction}


Here we present in detail a way to obtain the Marsden-Weinstein symplectic form for spaces of vortex membranes 
and vortex sheets from the Donaldson construction
of mappings to a surface, hinted in Remarks~\ref{rem:MW-filaments} and \ref{rem:MW-sheets}.
We start by using families of vortex membranes defined by a submersion
to obtain linear functionals on the space $\vect_\sigma(S)$ of (exact) divergence-free vector fields in $n$-dimensional $S$ with a volume form $\sigma$.

Let $\mathcal M:=\Sub(S, M)$ denote the space of proper submersions $f:(S,\sigma)\to (M, \omega)$, where
$M$ is a surface endowed with symplectic form $\omega$.
The 2-form $\Omega$  in Equation~\eqref{eq:Omega_f},
defined by  Donaldson's construction, is a symplectic form on  $\Sub(S, M)$.
The momentum map for the natural $\Diff_\sigma(S)$-action on $\Sub(S,M)$ is \cite{MW1983}
\begin{equation}\label{jsub}
J \,:\, \Sub(S,M)\to \vect_\sigma(S)^*,\quad
\langle J(f),v\rangle=\int_S\alpha\wedge f^*\omega,
\end{equation}
for a potential $(n-2)$-form $\alpha$ for the divergence-free vector field $v$.

Any proper submersion $f\in\Sub(S, M)$ determines a codimension 2 fibration of $S$  with compact fibers $f^{-1}(q)$ for $q\in M$.
Thus we get a smooth map
\[
\Sigma:\Sub(S,M)\to C^\infty(M,\mathcal C),\quad\Sigma(f)(q)=f^{-1}(q),
\]
which associates to a proper submersion $f$ a family of membranes parametrized by points of $M$.

Recall that $\mathcal C$, the space of codimension 2 closed oriented submanifolds of $S$,
can be endowed with the Marsden-Weinstein symplectic form $\Omega^{MW}$. 
Thus, $C^\infty(M,\mathcal C)$ carries a natural symplectic form, which can be regarded as a version of  Donaldson's construction
for maps into an infinite-dimensional symplectic manifold $(\mathcal C, \Omega^{MW})$:
\begin{equation}\label{boldo}
{\bf \Omega}_F(u_F,v_F)=\int_M\Omega^{MW}(u_F,v_F)\omega=\int_M(\int_{F(q)}i_{u_F(q)}i_{v_F(q)}\sigma)\,\omega(q),
\end{equation}
for $F\in C^\infty(M,\mathcal C)$ and tangent vectors $u_F,v_F\in \Gamma(F^*T\mathcal C)$,
which means that $u_F(q),v_F(q)\in T_{F(q)}\mathcal C=\Gamma(T^\perp F(q))$. 
That $C^\infty(M,\mathcal C)$ is also a smooth manifold follows by \cite[Section 42.4]{KM}.

We notice that the map $\Sigma$ takes values in the manifold of embeddings  $\Emb(M,\mathcal C)$.
Indeed, each $F:=\Sigma(f)$ is an embedding because the derivative of $F$ at $q$ maps a vector $w_q\in T_qM$ 
to the  element of $T_{f^{-1}(q)}\mathcal C$ given by the normal vector field along $f^{-1}(q)$ that is canonically defined by $w_q$.
In particular $T_qF$ is injective, so $F$ is an immersion. It's not difficult to see that $F$  is a homeomorphism onto its image
(with inverse $f^{-1}(q)\mapsto q$), so $F$ is also an embedding.

\begin{theorem}\label{PS}
The map $\Sigma:\Sub(S,M)\to \Emb(M,\mathcal C)$ is symplectic: 
$$
\Sigma^*{\bf \Omega}=\Omega.
$$
\end{theorem}

\begin{proof}
We start by computing  $T_f\Sigma(u_f)$ for $u_f\in
\Gamma(f^*TM)$.
At $x\in f^{-1}(q)$, the vector $u_f(x)\in T_qM$ can be identified  with a vector in $T^\perp_xf^{-1}(q)$ via the submersion $f$.
Under this identification, $u_f$ becomes an element of $T_f\Sigma(u_f)\in T_{\Sigma(f)}\Emb(M,\mathcal C)$, namely
\begin{equation}\label{ubar}
\bar u_f:q\in M\mapsto u_f|_{f^{-1}(q)}^\perp\in \Gamma(T^\perp f^{-1}(q)) = T_{f^{-1}(q)}\mathcal C.
\end{equation}
Now the result follows from the direct calculation:
\begin{align*}
(\Sigma^*{\bf \Omega})_f(u_f,v_f)&={\bf \Omega}_{\Sigma(f)}(\bar u_f,\bar v_f)
=\int_M(\int_{f^{-1}(q)}i_{u_f|_{f^{-1}(q)}^\perp}i_{v_f|_{f^{-1}(q)}^\perp}\sigma)\,\omega\\
&=\int_Si_{u_f}i_{v_f}\sigma\wedge f^*\omega=\Omega_f(u_f,v_f),
\end{align*}
 for $u_f,v_f\in\Gamma(f^*TM)$.
\end{proof}
\medskip

We define the averaging map over the symplectic manifold $(M, \omega)$ by
\begin{equation}\label{biga}
A: \Emb(M,\mathcal C)\to \overline\vect_\sigma(S)^*,\quad \langle A(F),X_\alpha\rangle=\int_M\left (\int_{F(q)}\alpha\right )\omega(q)\,,
\end{equation}
where $X_\alpha\in \vect_\sigma(S)$ is an exact vector field corresponding to a $(n-2)$-form $\alpha$: $\iota_{X_\alpha}\sigma=d\alpha$.
With the  $\Diff_\sigma(S)$-equivariant inclusion \eqref{flux}, rewritten as
\begin{equation}\label{smallj}
j:\mathcal C\to\overline\vect_\sigma(S)^*,\quad \langle j(\Sigma),X_\alpha\rangle=\int_\Sigma \alpha,
\end{equation}
the map \eqref{biga} takes the form $A(F)=\int_M(j\circ F)\omega$.

\begin{proposition}\label{LG}
The natural $\Diff_\sigma(S)$ action  on $(\Emb(M,\mathcal C),\bf\Omega)$ is Hamiltonian and the averaging map $A$ from \eqref{biga}
is a  $\Diff_\sigma(S)$-equivariant momentum map.
\end{proposition}

\begin{proof}
For the proof it is convenient to use the hat calculus from \cite{Vizman2011}. 
 (Recall, that the hat pairing takes a $p$-form on $P$ and a $q$-form on $Q$, and produces a $(p+q-k)$-form on the function manifold 
 $C^\infty(Q,P)$, where $k$ is the dimension of $Q$. It is defined using fiber integration (push-forward) over the domain manifold $Q$ via evaluation and projection maps.) 
 With this notation the symplectic form  in \eqref{boldo} becomes
${\bf \Omega}:=\widehat{\Omega^{MW}\cdot\omega}$. Moreover, the averaging map in \eqref{biga} becomes $A=\widehat{j\cdot \omega}$,
where the function $\langle\widehat{ j\cdot\omega},v\rangle$ on $\Emb(M,\mathcal C)$ is given by
$$
\langle\widehat{ j\cdot\omega},v\rangle(F):=\widehat{(\langle j,v\rangle\cdot\omega)}(F)= \langle A(F),v\rangle
$$
for every $v\in\vect_\sigma(S)$.
We write the infinitesimal generator of $v$ on $\Emb(M,\mathcal C)$ as $v_{\Emb}=\overline{ v_{\mathcal C}}$,
where $v_{\mathcal C}$ denotes the infinitesimal generator on $\mathcal C$,
meaning that $v_{\Emb}(F)=\overline{ v_{\mathcal C}}\circ F$.
The computation
\begin{align*}
i_{v_{\Emb}}{\bf\Omega}&=i_{\overline{ v_{\mathcal C}}}\widehat {(\Omega^{MW}\cdot\omega)}
=(i_{v_{\mathcal C}}\Omega^{MW}\cdot\omega)\widehat{\ }
={(d\langle j,v\rangle\cdot\omega)}\widehat{\ }
=d\widehat{(\langle j,v\rangle\cdot\omega)}\\
&=d\langle\widehat{ j\cdot\omega},v\rangle=d\langle A,v\rangle.
\end{align*}
shows that $A$ is a momentum map.

The  $\Diff_\sigma(S)$-equivariance of $A$ follows directly from the  $\Diff_\sigma(S)$-equivariance of $j$.
\end{proof}

Because $\Sigma$ is a symplectic $\Diff_\sigma(S)$-equivariant map and $A$ is a momentum map,
the composition $A\circ\Sigma$ is a momentum map too.
We discover that $A\circ\Sigma$ is exactly the momentum map  $J$ from \eqref{jsub}.
For this we compute 
\[
\langle J(f),v\rangle=\int_M\langle j\circ \Sigma(f),v\rangle\,\omega=\int_M\left(\int_{f^{-1}(q)}\alpha\right)\,\omega(q).
\]
Now, by the Fubini theorem for the (proper) submersion $f$,
we have that $\int_S\alpha\wedge f^*\omega=\int_M(\int_{f^{-1}(q)}\alpha)\,\omega(q)$.
Thus $J(f)$ is nothing else but the regular momentum characterized by the closed 2-form $f^*\omega$ on $S$.

In conclusion, the map $\Sigma$ provides a link between the $\vect_\sigma(S)^*$-valued momentum map $J$ and  the
$\overline\vect_\sigma(S)^*$-valued momentum map $A$:
$$
\begin{tikzcd}[row sep=large, column sep=large]
\Sub(S,M) \arrow[r, "\Sigma"] \arrow[d,"J" '] & \Emb(M,\mathcal C) \arrow[d,"A"] \\
 \vect_\sigma(S)^* \arrow[r,hook]           &  \overline\vect_\sigma(S)^*
\end{tikzcd}
$$
In turn,  the Lie-Poisson structure on $\overline\vect_\sigma(S)^*$ and on its special coadjoint orbits is responsible for the Marsden-Weinstein symplectic structure on the space of membranes $\mathcal C$, as discussed in Remark~\ref{rem:MW-filaments}. 

\begin{remark}
A similar diagram furnishes details for the symplectic structure on the space of vortex sheets, see Remark~\ref{rem:MW-sheets}. 
It holds when replacing $\Sigma$  by $\iota^*\Sigma:\Sub(S,M)\to \Emb(L,\mathcal C)$ for $\iota:L\hookrightarrow M$ 
and the averaging map $A$ by the momentum map
$$
A_{L,\lambda}:\Emb(L,\mathcal C)\to\overline\vect_\sigma(S)^*,\quad F\mapsto \int_L(j\circ F)\lambda.
$$
In particular, $J$ has to be replaced by the vortex sheet map
$$
J_{L,\lambda}:\Sub(S,M)\to\overline\vect_\sigma(S)^*,\quad f\mapsto (f^{-1}(L),f^*\lambda),
$$
which is a Poisson map, but not a momentum map.
\end{remark}

\bibliographystyle{amsplain}
\bibliography{bibliography}

@article{Do1999,
	author = {Donaldson, Simon K.},
	journal = {Asian J. Math.},
	number = {1},
	pages = {1--16},
	title = {Moment maps and diffeomorphisms},
	volume = {3},
	year = {1999}}

@article{We1990,
	author = {Weinstein, Alan},
	journal = {Adv. Math.},
	number = {2},
	pages = {133--159},
	title = {Connections of {B}erry and {H}annay type for moving {L}agrangian submanifolds},
	volume = {82},
	year = {1990}}

@article{MW1983,
title = {Coadjoint orbits, vortices, and {C}lebsch variables for incompressible fluids},
journal = {Physica D: Nonlinear Phenomena},
volume = {7},
number = {1},
pages = {305-323},
year = {1983},
issn = {0167-2789},
author = {Jerrold Marsden and Alan Weinstein}
}

@article{Kh2012,
  author    = {Khesin, Boris},
  title     = {Symplectic structures and dynamics on vortex membranes},
  journal   = {Moscow Mathematical Journal},
  volume    = {12},
  number    = {2},
  pages     = {413--434},
  year      = {2012}
}

@article{Shashi,
  author    = {Shashikanth,  Banavara N.},
  title     = {Vortex dynamics in ${R}^4$},
  journal   = {J. Math. Phys.},
  volume    = {53},
  number    = {1},
  pages     = {013103},
  year      = {2012}
}

@book{ArnoldKh,
  title={Topological Methods in Hydrodynamics},
  author={Arnold, Vladimir I. and Khesin, Boris A.},
  volume={125},
  year={2021},
  publisher={Springer},
  series={Applied Mathematical Sciences},
  address={Cham, Switzerland},
  edition={2nd},
  isbn={978-3-030-74277-5},
  doi={10.1007/978-3-030-74278-2}
}

@book{KM,
  title={The Convenient Setting of Global Analysis},
  author={Kriegl, Andreas and Michor, Peter W.},
  volume={53},
  year={1997},
  publisher={AMS},
  series={Mathematical Surveys and Monographs},
  address={},
  edition={first},
  isbn={},
  doi={}
}

@article{GBV12,
  author    = {Gay-Balmaz, Fran{\c{c}}ois and Vizman, Cornelia},
  title     = {Dual pairs in fluid dynamics},
  journal   = {Annals of Global Analysis and Geometry},
  year      = {2012},
  volume    = {41},
  number    = {1},
  pages     = {1--24},
  doi       = {10.1007/s10455-011-9267-z},
  url       = {https://doi.org/10.1007/s10455-011-9267-z}
}

@article{Calini,
  author    = {Calini, Annalisa and Ivey, Thomas},
  title     = {Finite-Gap Solutions of the Vortex Filament Equation: Isoperiodic Deformations},
  journal   = {Journal of Nonlinear Science},
  year      = {2007},
  volume    = {17},
  number    = {6},
  pages     = {527--567},
  doi       = {10.1007/s00332-007-9002-x},
  url       = {https://doi.org}
}

@article{Kh2012Gafa,
  author = {Khesin, Boris},
  title = {Dynamics of symplectic fluids and point vortices},
  journal = {Geometric and Functional Analysis},
  year = {2012},
  volume = {22},
  number = {5},
  pages = {1444--1459},
  doi = {10.1007/s00039-012-0182-7},
  url = {https://doi.org}
}

@article{Lee_2009,
	doi = {10.3842/sigma.2009.099},
	url = {https://doi.org},
	year = 2009,
	month = {nov},
	volume = {5},
	pages = {099},
	author = {Brian Lee},
	title = {Geometric Structures on Spaces of Weighted Submanifolds},
	journal = {SIGMA (Symmetry, Integrability and Geometry: Methods and Applications)}
}

@article{Vizman2011,
  author    = {Vizman, Cornelia},
  title     = {Induced differential forms on manifolds of functions},
  journal   = {Archivum Mathematicum},
  volume    = {47},
  number    = {3},
  pages     = {201--215},
  year      = {2011},
  url       = {http://eudml.org/doc/247143}
}

@article{HV2025,
  author    = {Haller, Stefan and Vizman, Cornelia},
  title     = {A Dual Pair for the Group of Volume Preserving Diffeomorphisms},
  journal   = {International Mathematics Research Notices},
  year      = {2025},
  volume    = {2025},
  number    = {13},
  pages     = {rnaf181},
  month     = {July},
  doi       = {10.1093/imrn/rnaf181},
  url       = {https://doi.org/10.1093/imrn/rnaf181}
}

@article{HV2004,
  author  = {Haller, Stefan and Vizman, Cornelia},
  title   = {Non-linear {G}rassmannians as coadjoint orbits},
  journal = {Mathematische Annalen},
  volume  = {329},
  number  = {4},
  pages   = {771--785},
  year    = {2004},
  note    = {non-abriged version in arXiv:math/0305089}
}

@article{GBV2024,
  author   = {Gay-Balmaz, Fran{\c{c}}ois and Vizman, Cornelia},
  title    = {Coadjoint orbits of vortex sheets in ideal fluids},
  journal  = {Journal of Geometry and Physics},
  volume   = {197},
  pages    = {105096},
  year     = {2024},
  month    = mar,
  doi      = {10.1016/j.geomphys.2023.105096}
}

\end{document}